\documentclass[11pt,reqno]{amsart}
\usepackage[
    letterpaper,
    left=1in,
    right=1in,
    top=1in,
    bottom=1in,
    headheight=14pt,
    headsep=18pt,
    footskip=28pt
]{geometry}
\usepackage[T1]{fontenc}
\usepackage{lmodern}
\usepackage[utf8]{inputenc}
\usepackage{microtype}
\usepackage{amsmath,amssymb,amsthm,mathtools}
\usepackage{mathrsfs}
\usepackage{bm}
\usepackage{booktabs}
\usepackage{tabularx}
\usepackage{array}
\usepackage{enumitem}
\usepackage[dvipsnames]{xcolor}
\usepackage[
    colorlinks=true,
    linkcolor=MidnightBlue,
    citecolor=OliveGreen,
    urlcolor=MidnightBlue
]{hyperref}
\newtheorem{theorem}{Theorem}[section]
\newtheorem{proposition}[theorem]{Proposition}
\newtheorem{lemma}[theorem]{Lemma}
\newtheorem{corollary}[theorem]{Corollary}
\theoremstyle{definition}

\theoremstyle{remark}
\newtheorem{remark}[theorem]{Remark}
\numberwithin{equation}{section}
\allowdisplaybreaks[2]

\setlist[enumerate]{
    leftmargin=2.2em,
    itemsep=0.25em,
    topsep=0.4em
}

\setlist[itemize]{
    leftmargin=2.0em,
    itemsep=0.2em,
    topsep=0.35em
}

\newcommand{\R}{\mathbb{R}}
\newcommand{\supp}{\operatorname{supp}}

\newcommand{\Div}{\operatorname{Div}}
\newcommand{\diag}{\operatorname{diag}}
\newcommand{\cM}{\mathcal{M}}
\newcommand{\1}{\mathbf{1}}
\newcommand{\dd}{\,d}

\newcommand{\tr}{\operatorname{tr}}

\newcommand{\norm}[1]{\lVert#1\rVert}
\title[GLOBAL WELL-POSEDNESS FOR THE VLASOV--RIESZ EQUATION]
{Global well-posedness for the four-dimensional repulsive
Vlasov--Poisson system and for Vlasov--Riesz equations}
\author{Emmanouil Katriadakis}
\date{}
\begin{document}
\begin{abstract}
We establish global classical well-posedness for the four-dimensional repulsive Vlasov--Poisson system, for every nonnegative
\(f_0\in C_c^1(\mathbb{R}^8)\). We also prove global well-posedness for Vlasov equations with power-law
interactions for \(d\ge4\). For every \(0<\alpha<3\),
we treat both attractive and repulsive interactions with arbitrary
nonnegative, compactly supported bounded initial data. Neither result requires a smallness
assumption on the initial data. For attractive interactions, a virial
argument yields finite-time breakdown for negative-energy solutions
when \(3\le\alpha<d+1\). The global existence proofs combine estimates along
characteristics, analysis in Lagrangian coordinates, and Serre's
compensated integrability.
\end{abstract}
\maketitle
\section{Introduction}
We consider the Vlasov equation
\begin{equation}
\partial_t f+v\cdot\nabla_x f
   +E_\sigma[f]\cdot\nabla_v f=0,
\qquad
\rho_f(t,x)
=
\int_{\mathbb R^d}f(t,x,v)\,dv,
\label{eq:intro-vlasov}
\end{equation}
with
\begin{equation}
E_\sigma[f]
=
\sigma K_\alpha*\rho_f,
\qquad
K_\alpha(x)
=
-c_{d,\alpha}\frac{x}{|x|^{\alpha+1}},
\qquad
c_{d,\alpha}>0.
\label{eq:intro-field}
\end{equation}
Here
\[
\sigma=+1
\quad\text{denotes attraction},
\qquad
\sigma=-1
\quad\text{denotes repulsion},
\]
and
\[
|K_\alpha(x)|
\simeq
|x|^{-\alpha}.
\]
\\ We also consider the repulsive Vlasov-Poisson system in four dimensions
\begin{equation}\label{eq:VP}
 \partial_t f+v\cdot\nabla_x f+E\cdot\nabla_v f=0,
 \qquad (x,v)\in\R^4\times\R^4,
\end{equation}
with
\begin{equation}\label{eq:E}
 \rho(t,x)=\int_{\R^4}f(t,x,v)\,\dd v,
 \qquad
 E(t,x)=\kappa\int_{\R^4}\frac{x-y}{|x-y|^4}\rho(t,y)\,\dd y,
 \qquad \kappa=\frac1{2\pi^2}.
\end{equation}

The starting point for this work is the classical global theory of the three-dimensional Vlasov--Poisson equation.  The papers of Pfaffelmoser~\cite{Pfaffelmoser}, Schaeffer~\cite{Schaeffer}, and Lions--Perthame~\cite{LionsPerthame} showed that smooth solutions arising from nonnegative smooth, compactly supported initial data. exist for all time.  These results are among the foundational global well-posedness theorems in kinetic theory.

A natural question is how the global theory changes when one moves beyond the three-dimensional Coulomb case, both by varying the singularity of the interaction and by considering higher dimensions.  Power-law potentials provide a natural setting in which to ask this question.

In recent years, several authors have made progress in this direction, primarily for small initial data.  For the Vlasov--Poisson equation, Hwang, Rendall, and Vel\'azquez~\cite{HwangRendallVelazquez} proved global small-data solutions with sharp decay estimates in dimensions \(d\ge3\).

More recently, Huang and Kwon~\cite{HuangKwon} proved small-data global well-posedness and scattering in three dimensions throughout the range $2<\alpha<3$ in the notation of the present paper.  Hong and Pankavich~\cite{HongPankavichShortRange,HongPankavichLongRange} have also obtained small-data global and scattering results for broader classes of non-Coulomb interaction potentials.  These results extend the global theory beyond the classical three-dimensional Coulomb problem, but largely under smallness assumptions. The corresponding large-data theory for general power-law interactions remains largely open. The present paper addresses this gap and establishes a large-data global theory in a substantial part of this regime.
\begin{theorem}[Global well-posedness with bounded data]
\label{thm:gwp}
Let \(d\ge4\) and \(0<\alpha<3\).  Assume
\[
f_0\in L^\infty(\mathbb R^{2d}),
\qquad
f_0\ge0,
\qquad
\operatorname{ess\,supp}f_0\Subset\mathbb R^{2d}.
\]
Then, for either sign \(\sigma\in\{-1,+1\}\), equation
\eqref{eq:intro-vlasov} admits a unique global bounded characteristic
solution.
\end{theorem}

\begin{theorem}\label{thm:main1}
Let $f_0\ge0$ belong to $ C_c^1(\R^{8})$. Then the repulsive four-dimensional Coulomb Vlasov--Poisson solution extends globally in time.
\end{theorem}

No smallness condition is imposed in either theorem.
In Theorem~\ref{thm:gwp}, the initial distribution is not
required to belong to a Sobolev or H\"older space.
The velocity-support estimates do not differentiate \(f\).
For \(0<\alpha<3\), the same estimates apply to the attractive
and repulsive equations, and no favorable sign of the interaction
energy is used. At the four-dimensional Coulomb endpoint,
repulsivity is used to obtain the first budget; the classical
continuation argument is supplied by
Proposition~\ref{audit:vp4-continuation}.
The upper endpoint in Theorem~\ref{thm:gwp} is independent of
the dimension. In dimension four, \(\alpha=3\) is the Coulomb
exponent; for \(d>4\), it lies strictly below the Coulomb exponent
\(d-1\).
\\For Theorem \ref{thm:main1}, the virial identity eliminates the possibility of global well-posedness in the attractive case with negative energy, so this result is fairly sharp.
\begin{remark}
The a priori argument is not specific to bounded characteristic solutions.
The same estimates apply to \(C_c^1(\mathbb R^{2d})\) solutions and to
compactly supported \(W^{1,\infty}\) solutions. Thus the same global
continuation mechanism applies in these regularity classes.
\end{remark}
\begin{theorem}[Virial obstruction]
\label{thm:virial-obstruction}
Let $d\ge4$ and $3\le\alpha<d+1$, and take the attractive sign
$\sigma=+1$. Let $f\ge0$ be a solution with finite mass,
kinetic energy, potential energy, and second spatial moment,
with enough regularity and decay to justify energy conservation
and the virial identity. For $d\le\alpha<d+1$, interpret the
force in the principal-value sense. If
\[
H_+=K-U<0,
\]
then the solution cannot exist globally in this class.
\end{theorem}
\subsection*{Relation with the classical Vlasov--Poisson theory}

The first classical approach, developed by
Pfaffelmoser~\cite{Pfaffelmoser} and
Schaeffer~\cite{Schaeffer}, is Lagrangian. The main object is the characteristic flow. They obtain bounds on the force integrated along characteristics, using the relative geometry of particle trajectories.  This gives control of the velocities and prevents the solution from breaking down.

The second approach, developed by Lions and Perthame~\cite{LionsPerthame}, is based on the propagation of velocity moments, which are converted into spatial integrability estimates for the density and the force field.  Thus the two classical approaches provide two complementary sources of control: the geometry of the characteristic flow on the one hand, and the propagation of moments and spatial integrability on the other.

The local and singularity theory for more singular interactions has also been studied. Bobylev, Dukes, Illner, and Victory~\cite{BDIV1,BDIV2} treated the Vlasov--Manev equation, while Choi and Jeong~\cite{ChoiJeong} developed a local theory for general Vlasov--Riesz interactions and established attractive singularity formation in several regimes.

Directly applying the estimates of Pfaffelmoser and Schaeffer does not suffice to close the characteristic argument throughout the range of Theorem~\ref{thm:gwp}.  Likewise, the estimates of Lions and Perthame do not provide enough control to establish global well-posedness in the full range \(0<\alpha<3\) and \(d\ge4\) or the repulsive four dimensional case. 

Our proof keeps the Lagrangian viewpoint of the Pfaffelmoser--Schaeffer theory, but adds compensated integrability, in the form developed by Serre~\cite{Serre2019}.  A related use of Serre's compensated-integrability estimate was recently made by Wang \cite{Wang2026} for the three-dimensional Vlasov--Poisson equation, where it yields the critical spacetime bound $\rho\in L^2_{t,x}.$ 

We apply compensated integrability to arbitrary subsets of the support and combine the resulting support-independent bounds with a level-set argument along characteristics. This localization is a key ingredient in the tail improvement that ultimately gives global control of the velocity support. This additional flexibility allows the characteristic argument to be closed throughout the range \(0<\alpha<3\). The same mechanism will close the repulsive four dimensional Vlasov--Poisson as well.
On the singularity side, Theorem ~\ref{thm:virial-obstruction} follows from the standard energy--virial argument. Non-global classical solutions for the attractive Vlasov--Poisson equation in higher dimensions were already obtained by Horst~\cite{Horst1982}. The virial computation introduces no new difficulty in the present setting.

\noindent\textbf{Organization of the paper.}
Section~\ref{sec2} defines the solution classes and gives
local well-posedness and continuation.
Section~3 treats the range $0<\alpha\le1$,
and Section~\ref{sec4} outlines the remaining argument.
Sections~5--9 develop the compensated-integrability,
first-budget, tail, and characteristic estimates.
Section~10 proves global well-posedness.
Section~11 gives the conservation laws and the attractive
virial obstruction. The exponent calculations are collected
in the appendix.

\medskip
\noindent\textbf{Acknowledgments.}
I am grateful to my advisor, Luis Silvestre, for many helpful
discussions concerning this work. I also thank Stanley Snelson for
introducing me to the Vlasov--Poisson equation.

\section{Bounded-data local theory and continuation}\label{sec2}

We first consider $d\ge4$ and $0<\alpha<3$.
Then $\alpha+1<d$, so both $K_\alpha$ and its first derivative
are locally integrable. The four-dimensional Coulomb case is
treated separately at the end of this section.

\paragraph{Solution class.}
A bounded characteristic solution on $[0,T]$ is a nonnegative
\[
f\in L^\infty([0,T]\times\mathbb R^{2d})
\]
with phase-space support contained in one compact set on $[0,T]$,
whose field satisfies
\[
E_\sigma[f]\in
L^\infty\bigl(0,T;W^{1,\infty}(\mathbb R^d)\bigr),
\]
and which obeys
\[
f(t)=f_0\circ\Phi_t^{-1}\qquad\text{almost everywhere}
\]
for every $t\in[0,T]$, where $\Phi_t$ is the characteristic flow
of $(v,E_\sigma[f])$ with $\Phi_0=\mathrm{Id}$.
A global bounded characteristic solution has these properties
on every finite time interval.

\paragraph{Support convention.}
For bounded data, $\operatorname{supp}$ always means essential
support. We choose a measurable representative that vanishes
outside this support. All support radii and all suprema over
initial labels are taken over the corresponding essential support.
\begin{lemma}[Compact-support field estimate]\label{lem:compact-field}
Let $R>0$ and $g\in L^\infty(\R^d)$ with $\operatorname{ess\,supp}g\subset B_R$.  Then
\[
 K_\alpha*g\in W^{1,\infty}(\R^d),
\]
and
\begin{equation}\label{eq:compact-field}
 \|K_\alpha*g\|_\infty+\|\nabla(K_\alpha*g)\|_\infty
 \le C_{d,\alpha,R}\|g\|_\infty.
\end{equation}
\end{lemma}

\begin{theorem}[Local bounded strong well-posedness]\label{thm:lwp}
Let $d\ge4$, $0<\alpha<3$, and let $f_0$ satisfy the hypotheses of Theorem~\ref{thm:gwp}.  There exists $T_0>0$ and a unique bounded characteristic solution on $[0,T_0]$.  Moreover, if
\[
 \operatorname{ess\,supp}f_0\subset B_{R_0}\times B_{P_0},
\]
then $T_0$ may be bounded below by a positive quantity depending only on
\[
 d,\alpha,\|f_0\|_\infty,R_0,P_0.
\]
On that interval the phase-space support is bounded in terms of the same quantities.
\end{theorem}

\begin{proof}
This is the bounded compact-support strong theory of Hauray--Jabin~\cite[Proposition~2]{HaurayJabin}.  Their theorem applies to force singularities $\alpha\le d-1$.  Here $d\ge4$ and $\alpha<3$, hence $\alpha<d-1$ when $d=4$ and $\alpha<d-1$ a fortiori when $d>4$.  Their quantitative statement gives the asserted lower bound for the local lifespan and the support bound.

For completeness, once a bounded compactly supported solution is available, the velocity-support bound implies
\[
 \|\rho(t)\|_\infty
 \le C_d\|f_0\|_\infty Q(t)^d.
\]
Together with Lemma~\ref{lem:compact-field}, this yields $E_\sigma\in W^{1,\infty}_x$ on every local slab.  Hence the characteristic flow is uniquely defined and bi-Lipschitz.  Since the phase-space vector field $(v,E_\sigma(t,x))$ has zero divergence, Liouville's formula gives measure preservation, and therefore
\[
 f(t)=f_0\circ\Phi_t^{-1},
 \qquad
 \|f(t)\|_p=\|f_0\|_p,
 \quad 1\le p\le\infty.
\]
\end{proof}

\begin{corollary}[Continuation by velocity support]\label{cor:continuation}
Let $[0,T_*)$ be the maximal bounded characteristic lifespan and define
\[
 Q(t):=1+\sup\{|v|:(x,v)\in\operatorname{ess\,supp}f(t)\}.
\]
If
\[
 Q_*:=\sup_{0\le t<T_*}Q(t)<\infty,
\]
then $T_*=\infty$.
\end{corollary}

\begin{proof}
Assume $T_*<\infty$.  If $R_0$ bounds the initial spatial support, then
\[
 \sup\{|x|:(x,v)\in\operatorname{ess\,supp}f(t)\}
 \le R_0+T_*Q_*=:R_*.
\]
Transport preserves $F_0:=\|f_0\|_\infty$, so every restart data $f(t_0)$ has the same $L^\infty$ norm and has support contained in the fixed ball $B_{R_*}\times B_{Q_*}$.  The quantitative part of Theorem~\ref{thm:lwp} therefore supplies a restart time $\delta>0$ depending only on
\[
 d,\alpha,F_0,R_*,Q_*,
\]
and not on $t_0<T_*$.  Choosing $t_0>T_*-\delta/2$ extends the solution beyond $T_*$, a contradiction.
\end{proof}
\begin{proposition}[Classical Coulomb local theory and continuation]\label{audit:vp4-continuation}
Let $f_0\ge0$ belong to $C_c^1(\mathbb R^8)$.
The repulsive four-dimensional Vlasov--Poisson system has a
unique maximal classical solution on $[0,T_*)$.
The local lifespan has a positive lower bound depending only
on the initial support and $\|f_0\|_{C^1}$.
If $T_*<\infty$, then
\[
\sup_{0\le t<T_*}
\sup\{|v|:(x,v)\in\operatorname{supp}f(t)\}=\infty.
\]
\end{proposition}

\begin{proof}
Local classical existence and uniqueness follow from the usual
characteristic iteration. Here the Coulomb kernel is locally
integrable, and one may write $\nabla E=K*\nabla\rho$ for compactly
supported $C^1$ densities. For densities supported in $B_R$,
\[
\|E\|_\infty+\|\nabla E\|_\infty
\le C_R\bigl(\|\rho\|_\infty+\|\nabla\rho\|_\infty\bigr).
\]
This gives the stated dependence of the local lifespan.
We prove the continuation assertion.

Suppose $T_*<\infty$ and the velocity support remains bounded.
Then the spatial support and $\|\rho(t)\|_\infty$ remain bounded
for $t<T_*$, while $\|\rho(t)\|_1=\|f_0\|_1$.

For $K(z)=\kappa z/|z|^4$, its derivative away from the origin is
\[
\partial_jK_i(z)
=\kappa\left(
\frac{\delta_{ij}}{|z|^4}
-4\frac{z_i z_j}{|z|^6}
\right),
\]
which has zero spherical mean. Its distributional derivative
also contains the local term $\delta_{ij}\delta_0/4$.
Splitting at radii $r$ and $1$, and subtracting $\rho(x)$ in
the inner ball, gives, for $0<r\le1$,
\[
\|\nabla E(t)\|_\infty
\le C\left(
\|\rho(t)\|_1+\|\rho(t)\|_\infty
+r\|\nabla\rho(t)\|_\infty
+\|\rho(t)\|_\infty\log\frac1r
\right).
\]
Taking $r=(1+\|\nabla\rho(t)\|_\infty)^{-1}$ yields
\[
\|\nabla E(t)\|_\infty
\le C\left(1+\log(1+\|\nabla\rho(t)\|_\infty)\right).
\]
Here and below $C$ is uniform for $t<T_*$.

The velocity-support bound gives
\[
\|\nabla\rho(t)\|_\infty
\le C\|\nabla_{x,v}f(t)\|_\infty.
\]
Differentiating the characteristic equations therefore gives
\[
D^+\log\left(1+\|\nabla_{x,v}f(t)\|_\infty\right)
\le C\left(
1+\log\left(1+\|\nabla_{x,v}f(t)\|_\infty\right)
\right),
\]
where $D^+$ is the upper right Dini derivative.
Gronwall's inequality bounds $\|f(t)\|_{C^1}$ uniformly for $t<T_*$.
The local theory then supplies a uniform restart time and extends
the solution beyond $T_*$, a contradiction.
\end{proof}

\section{The easy range \texorpdfstring{$0<\alpha\le1$}{0 < alpha <= 1}}

The threshold $\alpha=1$ is natural for the characteristic argument. In this range no compensated integrability or crossing estimate is needed.
\begin{lemma}[Near--far Riesz estimate]\label{lem:nearfar}
Let $0<\alpha<d$ and let $\rho\ge0$ belong to $L^1(\R^d)\cap L^\infty(\R^d)$. Then
\begin{equation}\label{eq:nearfar}
\|K_\alpha*\rho\|_\infty
\le C_{d,\alpha}\|\rho\|_1^{1-\alpha/d}\|\rho\|_\infty^{\alpha/d}.
\end{equation}
\end{lemma}

\begin{proof}
For every $R>0$ and every $x\in\R^d$,
\begin{align*}
|(K_\alpha*\rho)(x)|
&\le C\int_{|x-y|\le R}|x-y|^{-\alpha}\rho(y)\,dy
 +C\int_{|x-y|>R}|x-y|^{-\alpha}\rho(y)\,dy\\
&\le C\|\rho\|_\infty R^{d-\alpha}+C\|\rho\|_1R^{-\alpha}.
\end{align*}
If $\rho\not\equiv0$, choose $R=(\|\rho\|_1/\|\rho\|_\infty)^{1/d}$; the zero case is trivial.
\end{proof}

\begin{proposition}[Global existence in the easy range]\label{prop:easy}
Let $d\ge4$, $0<\alpha\le1$, and let $f_0$ satisfy the assumptions of Theorem~\ref{thm:gwp}. Then the maximal local solution is global. More precisely,
\begin{equation}\label{eq:easyQ}
Q(t)\le Q(0)+C\int_0^t Q(s)^\alpha\,ds.
\end{equation}
Consequently, after enlarging the data-dependent constant $C$,
\[
Q(t)\le
\begin{cases}
\bigl(Q(0)^{1-\alpha}+C(1-\alpha)t\bigr)^{1/(1-\alpha)},&0<\alpha<1,\\
Q(0)e^{Ct},&\alpha=1.
\end{cases}
\]
\end{proposition}

\begin{proof}
Transport preserves $F_0:=\|f_0\|_\infty$ and the total mass $M:=\iint f_0$. At any time for which the solution exists, its velocity support lies in a ball of radius $Q(t)$, so
\[
\|\rho(t)\|_\infty\le C_dF_0Q(t)^d.
\]
Lemma~\ref{lem:nearfar} therefore gives
\[
\|E_\sigma(t)\|_\infty
\le CM^{1-\alpha/d}\|\rho(t)\|_\infty^{\alpha/d}
\le CQ(t)^\alpha.
\]
For every characteristic,
\[
|V(t)|\le |v_0|+\int_0^t\|E_\sigma(s)\|_\infty\,ds.
\]
Taking the supremum gives~\eqref{eq:easyQ}. Scalar comparison with $y'=Cy^\alpha$ gives the displayed bounds. Hence $Q(t)$ cannot blow up at finite time, and Corollary~\ref{cor:continuation} gives global existence.
\end{proof}

\section{The hard range: roadmap}\label{sec4}
For the a priori estimates in Sections~5--9, we consider either
\[
d\ge4,\qquad 1<\alpha<3,\qquad \sigma\in\{-1,+1\},
\]
or the repulsive four-dimensional Coulomb case
\[
(d,\alpha,\sigma)=(4,3,-1).
\]
In the first case we use the bounded characteristic solution.
In the Coulomb case we use the classical solution from
Proposition~\ref{audit:vp4-continuation}.
The argument below applies to both settings.
Let $T_*$ denote the corresponding maximal lifespan, fix
$0<T<T_*$, and write $E_\sigma$ for the full force field.
Define\begin{equation}\label{eq:a-def}
a(z):=\int_0^T|E_\sigma(t,X(t;z))|\,dt,
\qquad
J(\lambda):=\int_{\{a\ge\lambda\}}f_0(z)a(z)\,dz.
\end{equation}
\medskip
\noindent\textbf{Proof architecture.}
The hard-range argument has three steps:
\begin{enumerate}[label=\textbf{\arabic*.}]
\item We first obtain the support-independent budget
\[
\int f_0a\le C_T.
\]
For $\alpha<3$ this follows from the master kinetic-tensor
estimate. In the four-dimensional Coulomb case we also use
the positive repulsive interaction stress.
In either case, $J(\lambda)\le C_T$.\item In dimensions $4\le d\le6$, only when the initial tail does not already close the characteristic argument, the master kinetic-tensor estimate applied to the level sets $\{a\ge\lambda\}$ amplifies the tail to
\[
J(\lambda)\le C_T\lambda^{-\beta}
\qquad\text{for some }\beta>3.
\]
\item A tail exponent $\beta$ closes the characteristic argument whenever
\[
(d-\alpha)(d+\beta+2)>(d-1)^2.
\]
\end{enumerate}
Thus the proof combines compensated integrability,
a tail bootstrap, and a characteristic crossing estimate.
Compensated integrability enters both the first budget
and the tail amplification.
\section{Serre's estimate and the master Lagrangian tensor bound}

\begin{theorem}[Divergence-controlled slab estimate]\label{thm:serre}
Let $H=(t_-,t_+)\times\R^d$ and let $S$ be a symmetric positive-semidefinite $(d+1)\times(d+1)$ tensor whose entries belong to $L^1(H)$.  Assume its row-wise divergence is a bounded vector-valued measure and its  normal traces $Se_0$ at $t=t_\pm$ are bounded measures.  Then
\begin{equation}\label{eq:serre}
 \int_H(\det S)^{1/d}
 \le C_d\Bigl(
 \|Se_0(t_-)\|_{\cM}
 +\|Se_0(t_+)\|_{\cM}
 +\|\Div S\|_{\cM(H)}
 \Bigr)^{1+1/d}.
\end{equation}
\end{theorem}

\begin{proof}
We use Serre's bounded-domain theorem~\cite[Theorem~2.2]{Serre2019}. Let $\widetilde S$ be the extension of $S$ by zero outside the slab.
Its divergence consists of the interior divergence and the two
temporal normal traces. Choose $0\le\chi_R\le1$ equal to one on
$B_R$, supported in $B_{2R}$, with $|\nabla\chi_R|\le C/R$.
The tensor $\chi_R\widetilde S$ is compactly supported in spacetime.
Apply Serre's theorem on a ball containing its support; its boundary
normal trace is zero. Since
\[
\operatorname{Div}(\chi_R\widetilde S)
=\chi_R\operatorname{Div}\widetilde S+\widetilde S\nabla\chi_R,
\]
we obtain
\[
\begin{aligned}
\int_H\chi_R^{1+1/d}(\det S)^{1/d}
&\le C_d\Bigl(
\|Se_0(t_-)\|_{\mathcal M}
+\|Se_0(t_+)\|_{\mathcal M}\\
&\qquad
+\|\operatorname{Div}S\|_{\mathcal M(H)}
+\frac{C}{R}\|S\|_{L^1(H)}
\Bigr)^{1+1/d}.
\end{aligned}
\]
Letting $R\to\infty$ and applying Fatou's lemma proves the claim.
\end{proof}\begin{lemma}[Kinetic determinant bound]\label{lem:det}
Let $g\in L^1(\R^d)\cap L^\infty(\R^d)$, $g\ge0$, have finite second moment, and satisfy $0\le g\le F$, with $F>0$. Define
\[
\rho_g=\int g,
\qquad
j_g=\int vg,
\qquad
\Pi_g=\int v\otimes vg,
\qquad
A_g=
\begin{pmatrix}
\rho_g&j_g^T\\
j_g&\Pi_g
\end{pmatrix}.
\]
Then
\begin{equation}\label{eq:det-bound}
\det A_g\ge c_dF^{-2}\rho_g^{d+3}.
\end{equation}
\end{lemma}

\begin{proof}
If $\rho_g=0$ there is nothing to prove. Set
\[
u:=\frac{j_g}{\rho_g},
\qquad
\Sigma:=\frac1{\rho_g}\int(v-u)\otimes(v-u)g(v)\,dv.
\]
The covariance is positive definite: if $\xi^T\Sigma\xi=0$ for some $\xi\ne0$, then $g$ is supported almost everywhere on the affine hyperplane $\xi\cdot(v-u)=0$, a Lebesgue-null set, contradicting $\rho_g>0$. The Schur complement gives
\[
\det A_g=\rho_g^{d+1}\det\Sigma.
\]
For the probability measure $d\mu=\rho_g^{-1}g\,dv$,
\[
\int(v-u)^T\Sigma^{-1}(v-u)\,d\mu
=\operatorname{tr}(\Sigma^{-1}\Sigma)=d.
\]
Markov's inequality gives
\[
\mu\bigl\{(v-u)^T\Sigma^{-1}(v-u)\le2d\bigr\}\ge\frac12.
\]
Let this be $\mathcal E$. Since $g\le F$,
\[
\frac{\rho_g}{2}\le F|\mathcal E|.
\]
The change of variables $v=u+\Sigma^{1/2}w$ gives
\[
|\mathcal E|=|B_{\sqrt{2d}}|\sqrt{\det\Sigma}=C_d\sqrt{\det\Sigma}.
\]
Hence $\det\Sigma\ge c_dF^{-2}\rho_g^2$, and~\eqref{eq:det-bound} follows.
\end{proof}

\begin{proposition}[Master estimate for a Lagrangian component]\label{prop:master}
Fix a solution slab $0\le t\le T<T_*$. Let $G\subset\supp f_0$ be Borel and define
\[
g_G(t,\Phi_t(z)):=f_0(z)\1_G(z).
\]
Set
\[
M_G:=\int_G f_0(z)\,dz,
\qquad
J_G:=\int_G f_0(z)a(z)\,dz,
\qquad
\rho_G(t,x):=\int g_G(t,x,v)\,dv.
\]
Then $g_G$ solves
\[
\partial_tg_G+v\cdot\nabla_xg_G+E_\sigma\cdot\nabla_vg_G=0
\]
in distributions, and
\begin{equation}\label{eq:master}
\int_0^T\!\int_{\R^d}\rho_G^{1+3/d}\,dx\,dt
\le
C_dF_0^{2/d}M_G^{1/d}\bigl((1+Q_0)M_G+J_G\bigr),
\end{equation}
where
\[
F_0:=\|f_0\|_\infty,
\qquad
Q_0:=\sup\{|v|:(x,v)\in\supp f_0\}.
\]
The constant in~\eqref{eq:master} is independent of $G$, $T$, and the velocity support on $[0,T]$.
\end{proposition}

\begin{proof}
If $M_G=0$, then $g_G=0$ almost everywhere and there is nothing to prove. We may also suppose $F_0>0$, since $F_0=0$ gives the zero solution.

Because the phase flow preserves Lebesgue measure, for every smooth compactly supported test function $\varphi$,
\[
\int g_G(t,z)\varphi(t,z)\,dz
=
\int_G f_0(z_0)\varphi(t,\Phi_t(z_0))\,dz_0.
\]
Differentiating only the smooth test function along the flow proves the distributional transport equation; in particular, the rough indicator $\1_G$ is never differentiated.

Define
\[
j_G=\int vg_G\,dv,
\qquad
\Pi_G=\int v\otimes vg_G\,dv,
\qquad
A_G:=
\begin{pmatrix}
\rho_G&j_G^T\\
j_G&\Pi_G
\end{pmatrix}.
\]
Pointwise in $(t,x)$,
\begin{equation}\label{eq:AG-positive}
A_G(t,x)
=
\int_{\R^d}
\binom{1}{v}\binom{1}{v}^{\!T}g_G(t,x,v)\,dv
\ge0.
\end{equation}
Thus $A_G$ is symmetric positive semidefinite. Because $0\le g_G\le f$ and the fixed solution slab has compact phase-space support, every entry of $A_G$ belongs to $L^1((0,T)\times\R^d)$ and hence defines a bounded measure.

Testing the weak transport equation with a velocity cutoff that equals one on the phase support, and then with $v_i$ times that cutoff, gives
\[
\partial_t\rho_G+\nabla_x\cdot j_G=0,
\qquad
\partial_tj_G+\Div_x\Pi_G=\rho_GE_\sigma.
\]
Therefore
\begin{equation}\label{eq:AG-divergence}
\Div_{t,x}A_G=\binom{0}{\rho_GE_\sigma},
\qquad
\|\Div A_G\|_{\cM}
=\int_0^T\!\int\rho_G|E_\sigma|=J_G.
\end{equation}
The equality on the right follows from the Lagrangian representation and Tonelli.

Along a characteristic,
\[
|V(t;z)|\le |v_0|+a(z)\le Q_0+a(z),
\]
so for every $0\le t\le T$,
\begin{equation}\label{eq:jG}
\|j_G(t)\|_1\le Q_0M_G+J_G.
\end{equation}
Moreover
\[
\int\rho_G(t,x)\psi(x)\,dx
=
\int_Gf_0(z)\psi(X(t;z))\,dz,
\]
\[
\int j_G(t,x)\psi(x)\,dx
=
\int_Gf_0(z)V(t;z)\psi(X(t;z))\,dz.
\]
Both are continuous in $t$ by dominated convergence; for the second formula use the integrable majorant $f_0(Q_0+a)\1_G$.

For complete trace bookkeeping, extend $A_G$ by zero outside the slab and denote the extension by $\widetilde A_G$. Distributionally on $\R_t\times\R_x^d$,
\begin{equation}\label{eq:AG-zero-extension}
\Div\widetilde A_G
=
\binom{0}{\rho_GE_\sigma}\1_{(0,T)}
+ A_G(0,\cdot)e_0\,\delta_{t=0}
- A_G(T,\cdot)e_0\,\delta_{t=T}.
\end{equation}
Hence the normal traces required in Theorem~\ref{thm:serre} are precisely
\[
A_G(t)e_0=(\rho_G(t),j_G(t))^Tdx,
\qquad t\in\{0,T\},
\]
and by~\eqref{eq:jG},
\begin{equation}\label{eq:AG-trace-bound}
\|A_G(t)e_0\|_{\cM}
\le M_G+Q_0M_G+J_G.
\end{equation}

Set
\[
C_G:=(1+Q_0)M_G+J_G,
\qquad
\eta:=\frac{C_G}{M_G},
\qquad
s=\eta t,
\]
and define
\[
S_G(s,x):=
\begin{pmatrix}
\eta\rho_G&j_G^T\\
j_G&\eta^{-1}\Pi_G
\end{pmatrix},
\qquad t=s/\eta.
\]
Since
\[
S_G=D_\eta A_GD_\eta,
\qquad
D_\eta=\diag(\eta^{1/2},\eta^{-1/2},\ldots,\eta^{-1/2}),
\]
we have $S_G\ge0$. A direct row-wise computation gives
\[
\Div_{s,x}S_G
=
\binom{0}{\eta^{-1}\rho_GE_\sigma}.
\]
After the change of variables $ds=\eta\,dt$,
\[
\|\Div S_G\|_{\cM}=J_G.
\]
At either endpoint,
\[
\|S_Ge_0\|_{\cM}
\le \eta M_G+Q_0M_G+J_G
\le2C_G.
\]
Thus every hypothesis of Theorem~\ref{thm:serre} has now been verified explicitly, and it yields
\[
\int_0^{\eta T}\!\int(\det S_G)^{1/d}
\le C_dC_G^{1+1/d}.
\]
Furthermore
\[
\det S_G=\eta^{1-d}\det A_G,
\qquad ds=\eta\,dt,
\]
so
\[
\int_0^{\eta T}\!\int(\det S_G)^{1/d}
=
\eta^{1/d}\int_0^T\!\int(\det A_G)^{1/d}.
\]
Since $\eta=C_G/M_G$,
\begin{equation}\label{eq:detAG-integral}
\int_0^T\!\int(\det A_G)^{1/d}
\le C_dM_G^{1/d}C_G.
\end{equation}
Finally, Lemma~\ref{lem:det}, applied pointwise to $g_G(t,x,\cdot)$ with $F=F_0$, gives
\[
(\det A_G)^{1/d}
\ge c_dF_0^{-2/d}\rho_G^{1+3/d}.
\]
Combining this with~\eqref{eq:detAG-integral} proves~\eqref{eq:master}.
\end{proof}
\section{The four dimensional repulsive case}
Set
\[
 M=\iint f_0,\qquad F=\norm{f_0}_\infty.
\]
The repulsive energy is
\[
 H=\frac12\iint |v|^2f(t,x,v)\,\dd x\dd v
 +\frac\kappa4\iint\frac{\rho(t,x)\rho(t,y)}{|x-y|^2}\,\dd x\dd y,
\]
and is conserved and nonnegative term by term.
Serre's repulsive interaction stress is
\begin{equation}
 S(t,x)=\frac\kappa2\int_0^1\int_{\R^4}
 \frac{z\otimes z}{|z|^4}
 \rho(t,x+(1-s)z)\rho(t,x-sz)\,\dd z\dd s.
\end{equation}
It satisfies
\begin{equation}
 S\ge0,
 \qquad \operatorname{div}_xS=-\rho E,
 \qquad \int\tr S\,\dd x=2U(t),
\end{equation}
where $U$ is the potential-energy term. Therefore
\begin{equation}
 \mathsf A=
 \begin{pmatrix}
 \rho&j^T\\
 j&\Pi+S
 \end{pmatrix}
\end{equation}
is positive semidefinite and divergence free in $(t,x)$. Moreover,
\begin{equation}\label{eq:traceA}
 \int_{\R^4}\tr\mathsf A(t,x)\,\dd x=M+2H,
\end{equation}
so the tensor is integrable on every finite slab.
For a tensor on $(0,T)\times\R^4$, we extend it by zero. The divergence of the extension contains the interior divergence and the two temporal first-column traces.
\\
We use the kinetic determinant estimate from Lemma~5.2.
\begin{proposition}[Base density bound]\label{prop:base}\label{audit:vp4-density}
On every classical slab,
\begin{equation}\label{eq:X0}
 X_0:=\int_0^T\int_{\R^4}\rho^{7/4}\,\dd x\dd t\le C(F,M,H).
\end{equation}
In particular the constant is independent of $T$ and of the velocity support.
\end{proposition}

\begin{proof}
Extend $\mathsf A$ by zero outside the slab. Since $\Div\mathsf A=0$ inside, only the two temporal traces remain. Their total variation is controlled by mass and kinetic energy:
\[
 \int |(\rho,j)|\,\dd x
 \le M+\int|j|\,\dd x
 \le M+\sqrt{2MH}.
\]
Also $\mathsf A$ dominates the kinetic tensor
\[
 \int(1,v)\otimes(1,v)f\,\dd v.
\]
Thus Lemma \ref{lem:det} and monotonicity of the determinant on positive semidefinite matrices give
\[
 (\det\mathsf A)^{1/4}\ge cF^{-1/2}\rho^{7/4}.
\]
Apply compensated integrability again.
\end{proof}

\section{The first accumulated-field budget}\label{sec7}

Set
\[
p_0:=1+\frac3d=\frac{d+3}{d},
\qquad
X_T:=\int_0^T\!\int\rho^{p_0},
\qquad
B_1(T):=\int f_0(z)a(z)\,dz.
\]
By measure preservation and Tonelli,
\begin{equation}\label{eq:B1-rep}
B_1(T)=\int_0^T\!\int_{\R^d}\rho(t,x)|E_\sigma(t,x)|\,dx\,dt.
\end{equation}

\begin{proposition}[First budget]\label{prop:firstbudget}
Assume $1<\alpha<3$. For every finite $T<T_*$,\begin{equation}\label{eq:firstbudget}
X_T\le C(1+T),
\qquad
B_1(T)\le C_T,
\end{equation}
with constants independent of the velocity support on $[0,T]$.
\end{proposition}

\begin{proof}
Apply Proposition~\ref{prop:master} with $G=\supp f_0$. Then $M_G=M$ and $J_G=B_1(T)$, so
\begin{equation}\label{eq:X-vs-B}
X_T\le C_0+C_1B_1(T).
\end{equation}
Since $|K_\alpha(x)|\le C|x|^{-\alpha}$,
\[
B_1(T)
\le
C\int_0^T\iint\frac{\rho(t,x)\rho(t,y)}{|x-y|^\alpha}\,dx\,dy\,dt.
\]
By the bilinear Hardy--Littlewood--Sobolev inequality~\cite{LiebLoss},
\[
B_1(T)\le C\int_0^T\|\rho(t)\|_{p_\alpha}^2\,dt,
\qquad
p_\alpha:=\frac{2d}{2d-\alpha}.
\]
Interpolate between $L^1_x$ and $L^{p_0}_x$:
\[
\|\rho(t)\|_{p_\alpha}
\le M^{1-\theta_\alpha}\|\rho(t)\|_{p_0}^{\theta_\alpha},
\qquad
\theta_\alpha:=\frac{\alpha(d+3)}{6d}.
\]
The key identity is
\[
\frac{2\theta_\alpha}{p_0}=\frac\alpha3.
\]
Hence H\"older in time gives
\begin{equation}\label{eq:B-vs-X}
B_1(T)\le C T^{1-\alpha/3}X_T^{\alpha/3}.
\end{equation}
Combining~\eqref{eq:X-vs-B} and~\eqref{eq:B-vs-X},
\[
X_T\le C_0+C_2T^{1-\alpha/3}X_T^{\alpha/3}.
\]
Since $\alpha<3$, Young's inequality absorbs the last term and gives $X_T\le C(1+T)$. Substitution into~\eqref{eq:B-vs-X} gives $B_1(T)\le C_T$.

All quantities are finite on the fixed solution slab before the estimate is applied: compact support gives bounded $E_\sigma$, hence $B_1(T)<\infty$. The resulting bounds are independent of that a priori support size.
\end{proof}
$0<\alpha<3$ is critical for this proof method, so for the repulsive four dimensional Vlasov--Poisson equation, we have to do something different. \\
In the four dimensional repulsive case we get: \\
Volume preservation of the characteristic flow gives
\begin{equation}\label{eq:B1}
 \int f_0a
 =\int_0^T\int\rho|E|\,\dd x\dd t.
\end{equation}
The bilinear Hardy--Littlewood--Sobolev inequality gives
\[
 \int\rho|E|\le C\norm{\rho}_{8/5}^2.
\]
Interpolating $L^1$ and $L^{7/4}$,
\[
 \norm{\rho}_{8/5}
 \le M^{1/8}\norm{\rho}_{7/4}^{7/8}.
\]
Hence Proposition \ref{prop:base} yields
\begin{equation}\label{eq:B1bound}
 \int f_0a\le C(F,M,H).
\end{equation}
Therefore
\begin{equation}\label{eq:j0}
 J_\lambda\le C,
\end{equation}
which is the starting exponent $j_0=0$.

In particular,
\begin{equation}\label{eq:initial-tail}
J(\lambda)\le B_1(T)\le C_T
\qquad\text{for every }\lambda>0.
\end{equation}
Thus the first-budget bound holds in both settings fixed in
Section \ref{sec4}. Its proof is the only part of the a priori argument
that differs between the two settings. The component,
tail-amplification, and characteristic estimates apply to both.
\section{Tail amplification in dimensions \texorpdfstring{$4,5,6$}{4, 5, 6}}

For $s>0$, write
\[
G_s:=\{z:a(z)\ge s\},
\qquad
M_s:=\int_{G_s}f_0(z)\,dz,
\qquad
J_s:=J(s)=\int_{G_s}f_0(z)a(z)\,dz,
\]
and let $\rho_s$ be the density of the transported component associated with $G_s$. Then
\begin{equation}\label{eq:MsJs}
M_s\le s^{-1}J_s.
\end{equation}

\begin{proposition}[Tail amplification]\label{prop:tailbootstrap}
In either setting fixed in Section~4, assume
\[
\frac{3d}{d+3}<\alpha\le3.
\]Suppose that for some $\beta\ge0$,
\begin{equation}\label{eq:tailhyp}
J_s\le C_{\beta,T}s^{-\beta}
\end{equation}
for all sufficiently large $s$. Then, for all sufficiently large $\lambda$,
\begin{equation}\label{eq:tailimprove}
J_\lambda\le C'_{\beta,T}\lambda^{-\mathcal F_{d,\alpha}(\beta)},
\end{equation}
where the explicit improvement exponent $\mathcal F_{d,\alpha}(\beta)$ is given in Appendix~\ref{app:exponents}, equation~\eqref{eq:app-Fdef}. It is decreasing in $\alpha$, increasing in $\beta$, and satisfies
\begin{equation}\label{eq:Fendpoint}
\mathcal F_{d,\alpha}(\beta)
\ge
\mathcal F_{d,3}(\beta)
=
\frac{d(d+1)\beta+3d^2+d-18}{18}.
\end{equation}
\end{proposition}

\begin{proof}
For sufficiently large $s$, Proposition~\ref{prop:master},~\eqref{eq:MsJs}, and $s\ge2(1+Q_0)$ give
\begin{equation}\label{eq:levelset-X}
X_s:=\int_0^T\!\int\rho_s^{p_0}
\le
C F_0^{2/d}s^{-1/d}J_s^{1+1/d}.
\end{equation}
Indeed, $(1+Q_0)M_s\le J_s/2$ for such $s$.

For $0<\vartheta\le1$, let $p_\vartheta,q_\vartheta$ be defined by
\[
\frac1{p_\vartheta}=1-\frac{3\vartheta}{d+3},
\qquad
\frac1{q_\vartheta}=\frac{d\vartheta}{d+3}.
\]
Interpolation between mass and~\eqref{eq:levelset-X} yields
\begin{equation}\label{eq:mixed}
\|\rho_s\|_{L_t^{q_\vartheta}L_x^{p_\vartheta}}
\le
C s^{-A_d(\vartheta)}J_s^{B_d(\vartheta)},
\end{equation}
where
\[
A_d(\vartheta):=1-\frac{d+2}{d+3}\vartheta,
\qquad
B_d(\vartheta):=1-\frac{2}{d+3}\vartheta.
\]
The interpolation algebra is recorded in Appendix~\ref{app:tail-algebra}.

Choose
\[
L=c_T\lambda^{1/\alpha},
\]
with $c_T>0$ sufficiently small. For large $\lambda$, $2Q_0\le L<\lambda$. Split the initial labels into $\{a<L\}$ and $\{a\ge L\}$ and denote the corresponding fields by $E_{<L}$ and $E_L$. If $a(z)<L$, then
\[
|V(t;z)|\le Q_0+a(z)<\frac32L,
\]
so
\[
\|\rho_{<L}(t)\|_\infty\le CF_0L^d,
\qquad
\|\rho_{<L}(t)\|_1\le M.
\]
Lemma~\ref{lem:nearfar} therefore gives
\begin{equation}\label{eq:lowfield}
\|E_{<L}(t)\|_\infty\le CL^\alpha.
\end{equation}

Using the exact Lagrangian splitting and~\eqref{eq:MsJs},
\begin{align*}
J_\lambda
&=\int_0^T\!\int\rho_\lambda|E_\sigma|\\
&\le C_TL^\alpha M_\lambda
 +\int_0^T\!\int\rho_\lambda|E_L|\\
&\le C_T\frac{L^\alpha}{\lambda}J_\lambda
 +\int_0^T\!\int\rho_\lambda|E_L|.
\end{align*}
Choosing $c_T$ so that the first term is at most $J_\lambda/2$ gives
\begin{equation}\label{eq:highhigh}
J_\lambda
\le
C\int_0^T\iint
\frac{\rho_\lambda(t,x)\rho_L(t,y)}{|x-y|^\alpha}
\,dx\,dy\,dt.
\end{equation}

Now take the unique interpolation parameter
\begin{equation}\label{eq:theta-tail}
\theta:=\frac{\alpha(d+3)}{3d}-1.
\end{equation}
These assumptions ensure $0<\theta<1$, and this choice satisfies\begin{equation}\label{eq:tail-exponents}
\frac1{p_\theta}+\frac1{p_0}+\frac\alpha d=2,
\qquad
\frac1{q_\theta}+\frac1{p_0}=\frac\alpha3\le1.
\end{equation}
The first identity is the bilinear HLS relation.
The second gives the nonnegative time-Hölder exponent
$1-\alpha/3$, which vanishes at the Coulomb endpoint.
Consequently,\[
J_\lambda
\le
C_T
\|\rho_\lambda\|_{L_t^{q_\theta}L_x^{p_\theta}}
\|\rho_L\|_{L_{t,x}^{p_0}}.
\]
Writing
\[
A:=A_d(\theta),
\qquad
D:=1-B_d(\theta)=\frac{2\theta}{d+3},
\]
we obtain from~\eqref{eq:mixed} and~\eqref{eq:levelset-X}
\begin{equation}\label{eq:key-tail-bootstrap}
J_\lambda^D
\le
C_T\lambda^{-A}L^{-1/(d+3)}J_L^{(d+1)/(d+3)}.
\end{equation}
If $J_\lambda=0$, the conclusion is trivial. Otherwise insert~\eqref{eq:tailhyp} and $L=c_T\lambda^{1/\alpha}$ into~\eqref{eq:key-tail-bootstrap}. The resulting power of $\lambda$ is precisely $\mathcal F_{d,\alpha}(\beta)$; its derivation, monotonicity, and endpoint form~\eqref{eq:Fendpoint} are verified in Appendix~\ref{app:tail-algebra}.

\end{proof}
\begin{corollary}[Strong tail in the remaining dimensions]\label{cor:strongtail}
In either setting fixed in Section \ref{sec4}, assume
\[
4\le d\le6,\qquad \frac{3d}{d+3}<\alpha\le3.
\]Then for every finite $T<T_*$ there exist $\beta>3$, $C_T$, and $\lambda_0$ such that
\begin{equation}\label{eq:strongtail}
J(\lambda)\le C_T\lambda^{-\beta},
\qquad
\lambda\ge\lambda_0.
\end{equation}
\end{corollary}

\begin{proof}
The first budget~\eqref{eq:initial-tail} gives the initial exponent $\beta_0=0$. By~\eqref{eq:Fendpoint},
\[
\mathcal F_{5,3}(0)=\frac{31}{9}>3,
\qquad
\mathcal F_{6,3}(0)=\frac{16}{3}>3.
\]
Thus one bootstrap step suffices for $d=5,6$. For $d=4$,
\[
\mathcal F_{4,3}(0)=\frac{17}{9},
\qquad
\mathcal F_{4,3}\!\left(\frac{17}{9}\right)=\frac{323}{81}>3.
\]
Since the improvement map is decreasing in $\alpha$ and increasing
in the incoming exponent, two steps suffice in dimension four,
including the Coulomb case $\alpha=3$.
This proves the stated tail bound.\end{proof}

\section{From tail decay to characteristic control}

Fix $z_*\in\operatorname{ess\,supp}f_0$ and let
\[
(X_*,V_*)(t)=\Phi_t(z_*).
\]
Write
\[
a_*:=a(z_*),
\qquad
A_T:=\sup_{z\in\supp f_0}a(z),
\qquad
\Lambda:=1+A_T.
\]
Since $|V(t;z)|\le |v_0|+a(z)$,
\begin{equation}\label{eq:Q-vs-A}
Q_T:=1+\sup_{0\le t\le T}\sup\{|v|:(x,v)\in\supp f(t)\}
\le C_0\Lambda.
\end{equation}
\begin{lemma}\label{lem:crossing}
Fix a characteristic $(X_z,V_z)(t)=\Phi_t(z)$ and a relative-velocity scale $\tau>0$. There is a partition of $[0,T]$ into at most
\begin{equation}\label{eq:partition-count}
N(z)\le1+C\frac{a_*+a(z)}{\tau}
\end{equation}
intervals such that, on every interval $I$ which meets
\[
\tau\le |V_z(t)-V_*(t)|<2\tau,
\]
one has, for every $r>0$,
\begin{equation}\label{eq:crossing}
\int_{I\cap\{|X_z-X_*|>r\}}
|X_z(t)-X_*(t)|^{-\alpha}\,dt
\le C_\alpha\tau^{-1}r^{1-\alpha}.
\end{equation}
\end{lemma}

\begin{proof}
Partition time according to increments of
\[
\delta_z(t):=\int_0^t
\bigl(|E_\sigma(s,X_*(s))|+|E_\sigma(s,X_z(s))|\bigr)\,ds
\]
of size $\tau/8$. This gives~\eqref{eq:partition-count}.

Suppose an interval contains $t_*$ with
\[
\tau\le|V_z(t_*)-V_*(t_*)|<2\tau.
\]
Throughout that interval the change in relative velocity is at most $\tau/8$. If
\[
e:=\frac{V_z(t_*)-V_*(t_*)}{|V_z(t_*)-V_*(t_*)|},
\]
then, after weakening the constant,
\[
(V_z(t)-V_*(t))\cdot e\ge\frac34\tau.
\]
Thus $(X_z(t)-X_*(t))\cdot e$ is monotone with speed at least $c\tau$. On the shell

$$
2^j r<|X_z-X_*|\le 2^{j+1}r,
$$

we have

$$
|(X_z-X_*)\cdot e|\lesssim 2^jr.
$$

Since

$$
\frac{d}{dt}\big((X_z-X_*)\cdot e\big)
=(V_z-V_*)\cdot e
$$

has size at least \(c\tau\), the time spent in the shell is at most

$$
C\frac{2^jr}{\tau}.
$$
 Summing gives
\[
\sum_{j\ge0}\frac{2^jr}{\tau}(2^jr)^{-\alpha}
\le C_\alpha\tau^{-1}r^{1-\alpha},
\]
because $\alpha>1$.
\end{proof}

\begin{proposition}\label{prop:tailclosure}
Assume that for some $\beta\ge0$ there exist $s_0$ and $C_T$ such that
\begin{equation}\label{eq:tail-closure-hyp}
J(s)\le C_Ts^{-\beta},
\qquad s\ge s_0.
\end{equation}
Define
\begin{equation}\label{eq:kappaW}
\kappa:=\frac{d-\alpha}{d-1},
\qquad
W_{d,\beta}(\alpha):=d-(d+\beta+2)\kappa.
\end{equation}
If
\begin{equation}\label{eq:tail-close-cond}
W_{d,\beta}(\alpha)<1,
\end{equation}
then
\[
A_T+Q_T\le C_T.
\]
Equivalently, the closing condition is
\begin{equation}\label{eq:tail-close-equivalent}
(d-\alpha)(d+\beta+2)>(d-1)^2.
\end{equation}
\end{proposition}

\begin{proof}
Choose
\[
P:=\Lambda^{1/(2\alpha)}.
\]
If $A_T$ is bounded by a sufficiently large fixed data-dependent constant, there is nothing to prove. Hence we may assume $A_T$ is large enough that
\[
P\ge2Q_0+2,
\qquad
P/2\ge s_0.
\]
At time $t$, define the good velocity region
\[
\mathcal G(t):=\{|w|<P\}\cup\{|w-V_*(t)|<P\}.
\]
Its velocity section is contained in the union of two balls of radius $P$, so the corresponding partial density satisfies
\[
\|\rho_{\mathcal G}(t)\|_\infty\le CF_0P^d,
\qquad
\|\rho_{\mathcal G}(t)\|_1\le M.
\]
Lemma~\ref{lem:nearfar} gives
\begin{equation}\label{eq:good}
I_{\mathcal G}
\le C_TP^\alpha
=C_T\Lambda^{1/2}.
\end{equation}
On the complement of $\mathcal G(t)$, use dyadic scales $\lambda,\tau$ with
\[
P\le\lambda,\tau\le C\Lambda
\]
and define
\[
S_{\lambda,\tau}(t)
:=
\{(y,w):\lambda\le|w|<2\lambda,
\ \tau\le|w-V_*(t)|<2\tau\}.
\]
Let $m:=\min\{\lambda,\tau\}$ and
\[
I_{\lambda,\tau}
:=
\int_0^T\iint_{S_{\lambda,\tau}(t)}
\frac{f(t,y,w)}{|X_*(t)-y|^\alpha}
\,dy\,dw\,dt.
\]
Split at a spatial radius $r>0$.
\[
I_{\lambda,\tau}
=
I_{\lambda,\tau}^{\mathrm{near}}(r)
+
I_{\lambda,\tau}^{\mathrm{far}}(r),
\]
where
\[
I_{\lambda,\tau}^{\mathrm{near}}(r)
:=
\int_0^T
\iint_{\substack{
S_{\lambda,\tau}(t)\\
|X_*(t)-y|\le r
}}
\frac{f(t,y,w)}
{|X_*(t)-y|^\alpha}
\,dy\,dw\,dt,
\]
and
\[
I_{\lambda,\tau}^{\mathrm{far}}(r)
:=
\int_0^T
\iint_{\substack{
S_{\lambda,\tau}(t)\\
|X_*(t)-y|>r
}}
\frac{f(t,y,w)}
{|X_*(t)-y|^\alpha}
\,dy\,dw\,dt.
\]
For the near part, the velocity section has volume $O(m^d)$. Indeed, if $\lambda\le\tau$ it is contained in the ball $\{|w|<2\lambda\}$, while if $\tau<\lambda$ it is contained in $\{|w-V_*(t)|<2\tau\}$. Hence in both cases its partial density is at most $CF_0m^d$. Therefore
\begin{equation}\label{eq:near-bin}
I_{\lambda,\tau}^{\rm near}(r)
\le C_Tm^dr^{d-\alpha}.
\end{equation}

For the far part, \[
F_\lambda
:=
\left\{
z\in \operatorname{supp} f_0:
\exists\, t\in[0,T]\ \text{such that}\ 
\lambda\le |V(t;z)|<2\lambda
\right\}.
\] Since $\lambda\ge2Q_0$,
\[
a(z)\ge\lambda-Q_0\ge\frac\lambda2
\qquad(z\in F_\lambda).
\]
The tail hypothesis gives
\begin{equation}\label{eq:F-tail}
\int_{F_\lambda}f_0a
\le J(\lambda/2)
\le C_T\lambda^{-\beta},
\qquad
\int_{F_\lambda}f_0
\le C_T\lambda^{-\beta-1}.
\end{equation}
Using the measure-preserving characteristic change of variables, Tonelli, Lemma~\ref{lem:crossing}, and~\eqref{eq:partition-count}, the far contribution is bounded by
\begin{align*}
I_{\lambda,\tau}^{\rm far}(r)
&\le
C\tau^{-1}r^{1-\alpha}
\int_{F_\lambda}f_0(z)
\left(1+\frac{A_T+a(z)}{\tau}\right)\,dz\\
&\le
C_T\tau^{-1}r^{1-\alpha}
\left(
\lambda^{-\beta-1}
+A_T\tau^{-1}\lambda^{-\beta-1}
+\tau^{-1}\lambda^{-\beta}
\right).
\end{align*}
Since $A_T\le\Lambda$ and $\lambda,\tau\le C\Lambda$,
\begin{equation}\label{eq:far-bin}
I_{\lambda,\tau}^{\rm far}(r)
\le
C_T\Lambda\lambda^{-\beta-1}\tau^{-2}r^{1-\alpha}.
\end{equation}

Balance~\eqref{eq:near-bin} and~\eqref{eq:far-bin}:
\[
m^dr^{d-\alpha}
=
\Lambda\lambda^{-\beta-1}\tau^{-2}r^{1-\alpha}.
\]
Since the powers of $r$ differ by $d-1$, choose
\[
r_{\lambda,\tau}
:=
\left(
\Lambda\lambda^{-\beta-1}\tau^{-2}m^{-d}
\right)^{1/(d-1)}.
\]
Then
\begin{equation}\label{eq:balanced-bin}
I_{\lambda,\tau}
\le
C_T\Lambda^\kappa
m^{d(1-\kappa)}
\lambda^{-(\beta+1)\kappa}
\tau^{-2\kappa}.
\end{equation}
The remaining step is purely exponent bookkeeping. The two cases $\tau\le\lambda$ and $\lambda<\tau$ lead to the same worst power; Appendix~\ref{app:characteristic-algebra} gives the complete calculation. Its conclusion is
\begin{equation}\label{eq:perbin}
I_{\lambda,\tau}
\le C_T\bigl(\Lambda^\kappa+\Lambda^{W_{d,\beta}(\alpha)}\bigr).
\end{equation}

There are $O((\log(2+\Lambda))^2)$ dyadic pairs $(\lambda,\tau)$. Combining the bad-region sum with~\eqref{eq:good} and taking the supremum over the reference characteristic gives
\begin{equation}\label{eq:A-sublinear}
A_T
\le
C_T\Bigl(
\Lambda^{1/2}+\Lambda^\kappa+\Lambda^{W_{d,\beta}(\alpha)}
\Bigr)
(\log(2+\Lambda))^2.
\end{equation}
Since $\alpha>1$,
\[
0<\kappa<1.
\]
Under~\eqref{eq:tail-close-cond},
\[
\gamma_0:=\max\left\{\frac12,\kappa,W_{d,\beta}(\alpha)\right\}<1.
\]
Choose $\gamma\in(\gamma_0,1)$. Because $\log^2(2+\Lambda)$ is dominated by every positive power of $\Lambda$,
\[
A_T\le C_T(1+A_T)^\gamma.
\]
This scalar sublinear inequality bounds $A_T$, and~\eqref{eq:Q-vs-A} bounds $Q_T$.

Finally,
\[
W_{d,\beta}(\alpha)<1
\iff
(d-\alpha)(d+\beta+2)>(d-1)^2,
\]
which proves~\eqref{eq:tail-close-equivalent}.
\end{proof}

\begin{remark}[Relation with accumulated-field moments]\label{rem:Bk}
If for some $k\ge1$ one knows
\[
B_k(T):=\int f_0a^k\le C_T,
\]
then on $\{a\ge s\}$,
\[
J(s)=\int_{\{a\ge s\}}f_0a
\le s^{1-k}B_k(T).
\]
Thus Proposition~\ref{prop:tailclosure} applies with $\beta=k-1$ and recovers the criterion
\[
(d-\alpha)(d+k+1)>(d-1)^2.
\]
The global proof below does not need to introduce $B_k$ for $k>1$.
\end{remark}
In Propositions~\ref{prop:tailbootstrap}
and~\ref{prop:tailclosure}, the output constants and thresholds
may depend on the exponent, constant, and lower threshold
of the incoming tail bound.
\begin{lemma}[Uniformity on bounded time intervals]\label{lem94}
Fix $\overline T<\infty$. In either setting fixed in
Section~\ref{sec4}, consider the first-budget estimate and
the finite tail iteration used in Section~\ref{sec10}, starting from
$\beta_0=0$.

All constants and lower thresholds in this argument,
including the resulting bound from
Proposition~\ref{prop:tailclosure}, may be chosen uniformly for
\[
0<T\le\overline T,\qquad T<T_*,
\]
depending only on $\overline T$, the fixed initial data,
$d$, and $\alpha$, and never on $A_T$ or $Q_T$.
\end{lemma}
\begin{proof}
For $1<\alpha<3$, use Proposition \ref{prop:firstbudget}.
In the Coulomb case, Proposition~\ref{audit:vp4-density}
and the first-budget estimate in Section \ref{sec7} give a bound
independent of $T$.

Choose $c_{\overline T}>0$ so that the low-accumulation term is
absorbed for every $T\le\overline T$.
At each tail step, enlarge the output threshold until
$c_{\overline T}\lambda^{1/\alpha}$ lies above both the component
threshold and the threshold for the incoming tail bound.
All these thresholds depend only on the fixed data, the parameters,
and $\overline T$. Only one step is used in dimensions five and six,
and two in dimension four, including the Coulomb case.

The characteristic estimate depends only on these tail bounds,
the initial data, and the time horizon. Its constants are therefore
uniform for $T\le\overline T$ as well.
\end{proof}
\section{Proof of global well-posedness}\label{sec10}

Define the first-tail threshold
\begin{equation}\label{eq:alpha1}
\alpha_1(d)
:=d-\frac{(d-1)^2}{d+2}
=\frac{4d-1}{d+2}.
\end{equation}
The two threshold comparisons used below are
\[
\alpha_1(d)\ge3\iff d\ge7,
\qquad
\alpha_1(d)>\frac{3d}{d+3}\quad(4\le d\le6).
\]
Their elementary algebra, together with the endpoint check for the strong-tail closure, is collected in Appendix~\ref{app:threshold-algebra}.

\begin{proof}[Proof of Theorem~\ref{thm:gwp}]
Let $[0,T_*)$ be the maximal lifespan.

If $0<\alpha\le1$, Proposition~\ref{prop:easy} gives $T_*=\infty$.

Now assume $1<\alpha<3$ and fix $T<T_*$. Proposition~\ref{prop:firstbudget} gives
\[
B_1(T)\le C_T,
\]
so~\eqref{eq:initial-tail} gives the tail bound $J(s)\le C_T$, i.e. Proposition~\ref{prop:tailclosure} applies with $\beta=0$ whenever
\[
(d-\alpha)(d+2)>(d-1)^2,
\]
or equivalently
\[
\alpha<\alpha_1(d).
\]
Hence $A_T+Q_T\le C_T$ in that range. By the first threshold comparison above, this already covers every $1<\alpha<3$ when $d\ge7$.

It remains to consider
\[
4\le d\le6,
\qquad
\alpha_1(d)\le\alpha<3.
\]
By the second threshold comparison above, this range lies strictly above $3d/(d+3)$, so Corollary~\ref{cor:strongtail} gives
\[
J(s)\le C_Ts^{-\beta}
\qquad\text{for some }\beta>3.
\]
Since $\beta>3$, the closing condition follows from the endpoint arithmetic in Appendix~\ref{app:threshold-algebra}. Proposition~\ref{prop:tailclosure} therefore gives
\[
A_T+Q_T\le C_T.
\]

Thus in every case $0<\alpha<3$ the velocity support is bounded on each finite interval by a constant independent of the a priori support size. If $T_*<\infty$, apply Lemma~\ref{lem94} with $\overline T=T_*$. The estimates above are then uniform for every $T<T_*$, so letting $T\uparrow T_*$ gives
\[
\sup_{0\le t<T_*}Q(t)<\infty.
\]
This contradicts Corollary~\ref{cor:continuation}. Hence $T_*=\infty$.
\end{proof}
\begin{proof}[Proof of Theorem~\ref{thm:main1}]Let $[0,T_*)$ be the maximal classical lifespan and fix $0<T<T_*$.

Proposition \ref{audit:vp4-density} and the first-budget estimate
in Section \ref{sec7} give $B_1(T)\le C(f_0)$.
At $(d,\alpha)=(4,3)$, two applications of Proposition \ref{prop:tailbootstrap} yield
\[
J(s)\le C_Ts^{-323/81}
\]
for all sufficiently large $s$.
Since $323/81>3$, the characteristic closing condition holds:
\[
(4-3)\left(4+\frac{323}{81}+2\right)>9.
\]

Proposition \ref{prop:tailclosure} therefore bounds the velocity support.

If $T_*<\infty$, fix $\overline T>T_*$.
Lemma \ref{lem94} makes this bound uniform for all $T<T_*$.
Proposition~\ref{audit:vp4-continuation} then extends the solution
beyond $T_*$, a contradiction. Hence $T_*=\infty$.
\end{proof}
\section{Smooth conservation laws and the virial obstruction}\label{sec:virial}

This section is logically independent of the bounded-data global-existence proof.  It records the smooth identities that explain the threshold $\alpha=3$. %
Assume $d\ge4$ and $1<\alpha<d+1$, and put
\[
q:=\alpha-1,\qquad
W_q(x):=\frac{c_{d,\alpha}}{q}|x|^{-q}.
\]
Thus $0<q<d$ and $K_\alpha=\nabla W_q$.
For $\alpha<d$, the force is an ordinary convolution.
For $d\le\alpha<d+1$, it is understood in the principal-value sense.
Throughout this section, $f\ge0$ has finite mass, kinetic energy,
potential energy, and second spatial moment, with enough regularity
and decay to justify the displayed conservation and virial identities.
\[
M:=\iint f(t,x,v)\,dx\,dv,
\qquad
j(t,x):=\int vf(t,x,v)\,dv.
\]
Transport and measure preservation imply
\[
M=\iint f_0,
\qquad
\|f(t)\|_{L^p}=\|f_0\|_{L^p},
\qquad 1\le p\le\infty.
\]

\begin{lemma}[Continuity equation]\label{lem:continuity}
The density and current satisfy
\[
\partial_t\rho+\nabla_x\cdot j=0.
\]
\end{lemma}

\begin{proof}
Integrate~\eqref{eq:intro-vlasov} in $v$. The $v$-divergence term vanishes.
\end{proof}

Define
\[
K(t):=\frac12\iint |v|^2f(t,x,v)\,dx\,dv,
\]
\[
U(t):=\frac12\iint W_q(x-y)\rho(t,x)\rho(t,y)\,dx\,dy,
\]
and
\[
H_\sigma:=K-\sigma U.
\]

\begin{proposition}[Energy conservation]\label{prop:energy}
On every interval on which the integrations below are justified,
\[
H_\sigma(t)=H_\sigma(0).
\]
\end{proposition}

\begin{proof}
First assume the solution is smooth. Let $E_0:=\nabla W_q*\rho$, so $E_\sigma=\sigma E_0$. Multiplying the Vlasov equation by $|v|^2/2$ and integrating by parts gives
\[
K'(t)=\sigma\int j\cdot E_0\,dx.
\]
Writing $\Phi=W_q*\rho$ and using Lemma~\ref{lem:continuity}, symmetry of $W_q$, and integration by parts in $x$,
\[
U'(t)
=\int\Phi\,\partial_t\rho\,dx
=\int\nabla\Phi\cdot j\,dx
=\int E_0\cdot j\,dx.
\]
Thus $K'-\sigma U'=0$.

\end{proof}

Define the spatial second moment
\[
I(t):=\iint |x|^2f(t,x,v)\,dx\,dv.
\]

\begin{proposition}[Virial identity]\label{prop:virial}
For every sufficiently regular solution,
\begin{equation}\label{eq:virial}
I''(t)=4H_\sigma+2\sigma(2-q)U(t).
\end{equation}
\end{proposition}

\begin{proof}
Integration by parts gives
\[
I'(t)=2\iint x\cdot v f\,dx\,dv,
\]
and hence
\[
I''(t)=4K(t)+2\int x\cdot E_\sigma(t,x)\rho(t,x)\,dx.
\]
By symmetry,
\[
\int x\cdot E_\sigma(x)\rho(x)\,dx
=\frac\sigma2\iint
(x-y)\cdot\nabla W_q(x-y)\rho(x)\rho(y)\,dx\,dy.
\]
Since $W_q$ is homogeneous of degree $-q$,
\[
z\cdot\nabla W_q(z)=-qW_q(z),
\]
so
\[
\int x\cdot E_\sigma\rho\,dx=-\sigma qU(t).
\]
Therefore
\[
I''=4K-2\sigma qU.
\]
Since $K=H_\sigma+\sigma U$, this is~\eqref{eq:virial}.
\end{proof}

\begin{proof}[Proof of Theorem~\ref{thm:virial-obstruction}]
For attraction, $\sigma=+1$, so
\[
I''(t)=4H_++2(2-q)U(t).
\]
If $q\ge2$, then $2-q\le0$, while $U(t)\ge0$. If $H_+<0$,
\[
I''(t)\le4H_+<0.
\]
Thus
\[
I(t)\le I(0)+I'(0)t+2H_+t^2,
\]
whose right-hand side becomes negative for large $t$, contradicting $I(t)\ge0$.
\end{proof}

\clearpage
\appendix
\section{Exponent calculations}\label{app:exponents}

This appendix contains the elementary exponent manipulations stripped away from the main argument. 

\subsection{The first-budget interpolation}\label{app:first-budget-algebra}

Recall
\[
p_0=\frac{d+3}{d},
\qquad
p_\alpha=\frac{2d}{2d-\alpha}.
\]
To interpolate $L^{p_\alpha}$ between $L^1$ and $L^{p_0}$, choose $\theta_\alpha$ from
\[
\frac1{p_\alpha}=1-\theta_\alpha+\frac{\theta_\alpha}{p_0}.
\]
Since
\[
1-\frac1{p_\alpha}=\frac{\alpha}{2d},
\qquad
1-\frac1{p_0}=\frac3{d+3},
\]
we obtain
\begin{equation}\label{eq:app-first-theta}
\theta_\alpha=\frac{\alpha(d+3)}{6d}.
\end{equation}
Hence
\begin{equation}\label{eq:app-first-key}
\frac{2\theta_\alpha}{p_0}
=2\frac{\alpha(d+3)}{6d}\frac{d}{d+3}
=\frac\alpha3.
\end{equation}
Therefore
\[
\|\rho(t)\|_{p_\alpha}^2
\le
M^{2(1-\theta_\alpha)}\|\rho(t)\|_{p_0}^{2\theta_\alpha},
\]
and H\"older in time gives
\[
\int_0^T\|\rho(t)\|_{p_\alpha}^2\,dt
\le
C T^{1-\alpha/3}
\left(\int_0^T\|\rho(t)\|_{p_0}^{p_0}\,dt\right)^{\alpha/3}.
\]
This is the exponent calculation behind~\eqref{eq:B-vs-X}. In particular, the absorption is subcritical exactly for $\alpha<3$.

\subsection{Tail-bootstrap exponents}\label{app:tail-algebra}

For $0<\vartheta\le1$, define
\[
\frac1{p_\vartheta}=1-\frac{3\vartheta}{d+3},
\qquad
\frac1{q_\vartheta}=\frac{d\vartheta}{d+3}.
\]
Since $p_0=(d+3)/d$, the second identity is equivalent to
\[
q_\vartheta=\frac{p_0}{\vartheta}.
\]
At each time,
\[
\|\rho_s(t)\|_{p_\vartheta}
\le
M_s^{1-\vartheta}\|\rho_s(t)\|_{p_0}^{\vartheta}.
\]
Taking the $L_t^{q_\vartheta}$ norm and using $q_\vartheta\vartheta=p_0$ gives
\[
\|\rho_s\|_{L_t^{q_\vartheta}L_x^{p_\vartheta}}
\le
M_s^{1-\vartheta}X_s^{\vartheta/p_0}.
\]
Now insert
\[
M_s\le s^{-1}J_s,
\qquad
X_s\le Cs^{-1/d}J_s^{1+1/d}.
\]
The power of $s$ is
\[
-(1-\vartheta)-\frac{\vartheta}{dp_0}
=-1+\frac{d+2}{d+3}\vartheta,
\]
while the power of $J_s$ is
\[
1-\vartheta+\frac{(1+1/d)\vartheta}{p_0}
=1-\frac{2\vartheta}{d+3}.
\]
Thus
\begin{equation}\label{eq:app-AB}
A_d(\vartheta)=1-\frac{d+2}{d+3}\vartheta,
\qquad
B_d(\vartheta)=1-\frac{2\vartheta}{d+3},
\end{equation}
which yields~\eqref{eq:mixed}.

We next choose $\vartheta$ so that the spatial bilinear HLS relation is exact:
\[
\frac1{p_\vartheta}+\frac1{p_0}+\frac\alpha d=2.
\]
Substituting the definitions gives
\[
1-\frac{3\vartheta}{d+3}+\frac d{d+3}+\frac\alpha d=2,
\]
so
\begin{equation}\label{eq:app-tail-theta}
\vartheta=\theta:=\frac{\alpha(d+3)}{3d}-1.
\end{equation}
Then
\[
\frac1{q_\theta}+\frac1{p_0}
=\frac{d\theta}{d+3}+\frac d{d+3}
=\frac\alpha3.
\]
Thus $\theta>0$ is equivalent to $\alpha>3d/(d+3)$, while $\theta<1$ follows from $\alpha\le3$ for $d\ge4$.

Set
\[
A:=A_d(\theta),
\qquad
D:=1-B_d(\theta)=\frac{2\theta}{d+3}.
\]
From~\eqref{eq:key-tail-bootstrap}, the hypothesis $J_L\lesssim L^{-\beta}$, and $L=c_T\lambda^{1/\alpha}$, we get
\[
J_\lambda^D
\lesssim_T
\lambda^{-A}
\lambda^{-1/[\alpha(d+3)]}
\lambda^{-\beta(d+1)/[\alpha(d+3)]}.
\]
Therefore
\begin{equation}\label{eq:app-Fdef}
\boxed{
\mathcal F_{d,\alpha}(\beta)
=
\frac{A+\dfrac{1+(d+1)\beta}{\alpha(d+3)}}{D}
}.
\end{equation}
This is the exponent in Proposition~\ref{prop:tailbootstrap}.

For monotonicity in $\alpha$, direct differentiation gives
\begin{equation}\label{eq:app-Falpha}
\partial_\alpha\mathcal F_{d,\alpha}(\beta)
=-\frac{3d\,\mathcal N_{d,\alpha,\beta}}
{2\alpha^2\bigl(\alpha(d+3)-3d\bigr)^2},
\end{equation}
where
\[
\mathcal N_{d,\alpha,\beta}
=
\alpha^2(d+3)^2+2\alpha(d+3)-3d
+\beta\bigl(2\alpha d^2+8\alpha d+6\alpha-3d^2-3d\bigr).
\]
At $\alpha=3d/(d+3)$, the $\beta$-independent part is $9d^2+3d>0$ and the coefficient of $\beta$ is $3d(d+1)>0$; both increase with $\alpha$. Hence $\mathcal N_{d,\alpha,\beta}>0$ throughout the admissible range and
\[
\partial_\alpha\mathcal F_{d,\alpha}(\beta)<0.
\]
Moreover
\begin{equation}\label{eq:app-Fbeta}
\partial_\beta\mathcal F_{d,\alpha}(\beta)
=\frac{d+1}{\alpha(d+3)D}>0.
\end{equation}
At $\alpha=3$, formula~\eqref{eq:app-Fdef} simplifies to
\begin{equation}\label{eq:app-Fendpoint}
\mathcal F_{d,3}(\beta)
=
\frac{d(d+1)\beta+3d^2+d-18}{18}.
\end{equation}
In particular,
\[
\mathcal F_{4,3}(0)=\frac{17}{9},
\qquad
\mathcal F_{4,3}\!\left(\frac{17}{9}\right)=\frac{323}{81}>3,
\]
\[
\mathcal F_{5,3}(0)=\frac{31}{9}>3,
\qquad
\mathcal F_{6,3}(0)=\frac{16}{3}>3.
\]
This proves that one bootstrap step suffices in dimensions $5,6$, while two suffice in dimension $4$.

\subsection{Characteristic-bin powers}\label{app:characteristic-algebra}

After balancing the near and far contributions, the main proof obtains
\begin{equation}\label{eq:app-balanced}
I_{\lambda,\tau}
\le
C_T\Lambda^\kappa
m^{d(1-\kappa)}
\lambda^{-(\beta+1)\kappa}
\tau^{-2\kappa},
\qquad
m=\min\{\lambda,\tau\},
\end{equation}
where
\[
\kappa=\frac{d-\alpha}{d-1}.
\]
We verify the uniform per-bin bound in the two possible orderings of the dyadic scales.

If $\tau\le\lambda$, then $m=\tau$, and~\eqref{eq:app-balanced} becomes
\[
I_{\lambda,\tau}
\le
C_T\Lambda^\kappa
\lambda^{-(\beta+1)\kappa}\tau^a,
\qquad
a:=d-(d+2)\kappa.
\]
If $a\le0$, then $\lambda,\tau\ge1$ give $I_{\lambda,\tau}\lesssim_T\Lambda^\kappa$. If $a>0$, use $\tau\le\lambda$ to obtain
\[
I_{\lambda,\tau}
\le C_T\Lambda^\kappa\lambda^b,
\qquad
b:=d-(d+\beta+3)\kappa.
\]
If $b\le0$, this is again $O(\Lambda^\kappa)$. If $b>0$, then $\lambda\lesssim\Lambda$ and
\[
I_{\lambda,\tau}
\lesssim_T
\Lambda^{\kappa+b}
=
\Lambda^{d-(d+\beta+2)\kappa}.
\]

If $\lambda<\tau$, then $m=\lambda$, and
\[
I_{\lambda,\tau}
\le
C_T\Lambda^\kappa\lambda^c\tau^{-2\kappa},
\qquad
c:=d-(d+\beta+1)\kappa.
\]
If $c\le0$, again $I_{\lambda,\tau}\lesssim_T\Lambda^\kappa$. If $c>0$, use $\lambda<\tau$:
\[
I_{\lambda,\tau}
\le
C_T\Lambda^\kappa\tau^{c-2\kappa}
=C_T\Lambda^\kappa\tau^b,
\]
where the same $b=d-(d+\beta+3)\kappa$ appears. The preceding dichotomy therefore applies unchanged.

Thus, with
\begin{equation}\label{eq:app-W}
W_{d,\beta}(\alpha):=d-(d+\beta+2)\kappa,
\end{equation}
we obtain
\[
I_{\lambda,\tau}
\le C_T\bigl(\Lambda^\kappa+\Lambda^{W_{d,\beta}(\alpha)}\bigr),
\]
which is~\eqref{eq:perbin}. Finally,
\[
W_{d,\beta}(\alpha)<1
\iff
(d-\alpha)(d+\beta+2)>(d-1)^2,
\]
because $\kappa=(d-\alpha)/(d-1)$.

\subsection{Dimension and threshold arithmetic}\label{app:threshold-algebra}

For the initial tail $\beta=0$, the closing condition is
\[
(d-\alpha)(d+2)>(d-1)^2.
\]
Solving for $\alpha$ gives
\[
\alpha<\alpha_1(d),
\qquad
\alpha_1(d)
=d-\frac{(d-1)^2}{d+2}
=\frac{4d-1}{d+2}.
\]
Moreover,
\[
\alpha_1(d)\ge3
\iff
4d-1\ge3d+6
\iff
d\ge7.
\]
For $4\le d\le6$,
\begin{align*}
\alpha_1(d)-\frac{3d}{d+3}
&=\frac{(4d-1)(d+3)-3d(d+2)}{(d+2)(d+3)}\\
&=\frac{d^2+5d-3}{(d+2)(d+3)}>0.
\end{align*}
Hence whenever the initial tail fails to close in dimensions $4,5,6$, the exponent automatically lies in the tail-bootstrap range $\alpha>3d/(d+3)$.

After the bootstrap we have $\beta>3$.
For $4\le d\le6$ and $\alpha\le3$,
\[
(d-\alpha)(d+\beta+2)
>
(d-3)(d+5)
=
(d-1)^2+4(d-4)
\ge
(d-1)^2.
\]
The first inequality is strict because $d-\alpha\ge d-3>0$
and $\beta>3$. Thus the characteristic closing condition
also holds at $(d,\alpha)=(4,3)$.

\end{document}